\documentclass[11pt,a4paper]{amsart}
\usepackage{amssymb,amsmath,amsthm}
\usepackage{commath}
\usepackage{mathtools}

\usepackage{accents}

\usepackage[utf8]{inputenc}

\makeatletter
\def\mathcenterto#1#2{\mathclap{\phantom{#1}\mathclap{#2}}\phantom{#1}}
\let\old@widetilde\widetilde
\def\widetildeto#1#2{\mathcenterto{#2}{\old@widetilde{\mathcenterto{#1}{#2\,}}}}
\makeatother

\newtheorem{theorem}{Theorem}
\newtheorem*{theorem*}{Theorem}

\newtheorem{conj}{Conjecture}

\newtheorem{prop}[theorem]{Proposition}

\renewcommand{\geq}{\geqslant}
\renewcommand{\leq}{\leqslant}

\newcommand{\C}{\mathbb{C}}
\newcommand{\R}{\mathbb{R}}
\newcommand{\eps}{\varepsilon}
\newcommand{\la}{\langle}
\newcommand{\ra}{\rangle}
\renewcommand{\u}{\widetilde{u}}
\newcommand{\vv}{\widetilde{v}}

\newcommand{\PP}{\mathbb{P}}
\renewcommand{\O}{\mathcal{O}}
\newcommand{\tr}{\mathrm{tr}}
\newcommand{\om}{\omega}

\newcommand{\U}{\widetilde{U}}

\newcommand{\M}{\widetildeto{N}{M}}

\DeclarePairedDelimiter\inner{\langle}{\rangle}
\DeclareMathOperator{\sgn}{sgn}
\DeclareMathOperator{\E}{\mathbb{E}}
\DeclareMathOperator{\Beta}{B}

\DeclareMathOperator{\im}{Im}

\begin{document}

\title[Polarization, sign sequences and isotropic vector systems]{Polarization, sign sequences and isotropic vector systems }

\author{Gergely Ambrus}
\address{
Gergely Ambrus\\
Alfr\'ed R\'enyi Institute of Mathematics, Hungarian Academy of Sciences\\
Reáltanoda u. 13-15\\
1053 Budapest\\
Hungary}
\email[G. Ambrus]{ambrus@renyi.hu}

\author{Sloan Nietert}
\address{
Sloan Nietert\\
Department of Computer Science\\
Cornell University\\
Ithaca, NY 14853-7501\\
U.S.A.
}
\email[S. Nietert]{sbn45@cornell.edu}

\thanks{Research of the first author was supported by NKFIH grants PD125502 and K116451 and by the Bolyai Research Scholarship of the Hungarian Academy of Sciences. Research of the second author was supported by Budapest Semesters in Mathematics and the Hungarian - American Fulbright Commission. }

\maketitle

\begin{abstract}
We determine the order of magnitude of the $n$th $\ell_p$-polari{\-}za{\-}tion constant of the unit sphere $S^{d-1}$ for every $n,d \geq 1$ and $p>0$. For $p=2$, we prove that extremizers are isotropic vector sets, whereas for $p=1$, we show that the polarization problem is equivalent to that of maximizing the norm of signed vector sums. Finally, for $d=2$, we discuss the optimality of equally spaced configurations on the unit circle.
\end{abstract}

\section{Introduction}

Let $\om_n =  \{u_1, \ldots, u_n \}$ be a multiset of $n$ unit vectors in $\R^d$, and set $p>0$. The {\em $\ell_p$-potential} of $\om_n$ at the unit vector $v \in S^{d-1}$ is defined as
\begin{equation*} 
U^p (\om_n , v ) = \sum_{i=1}^n |\la v, u_i \ra|^p,
\end{equation*}
where $\la \cdot,\cdot \ra$ denotes the standard inner product. This is an analogue of the classical Riesz potential for inner products.  The  {\em $\ell_p$-polarization of} $\om_n$ is given by
\begin{equation*} 
M^p (\om_n) = \max_{v \in S^{d-1}} U^p (\om_n , v ).
\end{equation*}

We are interested in finding the minimum $\ell_p$-polarization of $\om_n \subset S^{d-1}$, for fixed $d$ and $n$, that is,
\begin{equation*}
M^p_n(S^{d-1}) = \min_{\om_n  \subset S^{d-1}} M^p (\om_n) = \min_{u_1, \ldots, u_n \in S^{d-1}}\max_{v \in S^{d-1}} \sum_{i=1}^n |\la v, u_i \ra|^p.
\end{equation*}
The quantity $M^p_n(S^{d-1})$ is called the {\em $n$th $\ell_p$-polarization (or Chebyshev) constant of $S^{d-1}$.}

Related questions for $p \leq 0$ have been studied extensively, see e.g. the recent article of Hardin, Petrache and Saff~\cite{HPS19} about general polarization problems.  In the planar case, $M^p_n(S^{1})$ has a direct connection to the classical notions of Riesz potentials and Chebyshev constants. This connection is described in Section~\ref{sec_planar}. Polarization problems have been subject to very active research in the last 15 years, although their study dates back to at least 1967~\cite{O67}.
The most relevant results to our present problem are discussed in~\cite{A09, NR12, NR13, S75a, S75b}.

Determining the exact value of $M^p_n(S^{d-1})$  is hopeless in general, except for certain cases. Therefore, our first result provides asymptotic bounds. For brevity, we introduce the quantity
\begin{equation*}
\mu_{d,p} = \frac{\Gamma\left(\frac{d}{2}\right)\Gamma\left(\frac{p+1}{2}\right)}{\sqrt{\pi}\,\Gamma\left(\frac{d+p}{2}\right)}\, .
\end{equation*}
Clearly, $\mu_{d,p} = \Theta(d^{-p/2})$. Here, and throughout the paper, we are going to use the standard asymptotic notations following Knuth \cite{K97}: given two positive-valued functions $f(n) > 0$ and $g(n)>0$, $n \in \mathbb{N}$, we write
\begin{alignat*}{2}
 f(n) &= \O(g(n)) &&\textrm{ if } \limsup_{n \rightarrow \infty} f(n)/g(n) < \infty; \\
 f(n) &= o(g(n)) &&\textrm{ if } \lim_{n \rightarrow \infty} f(n)/g(n) =0; \\
f(n) &= \Omega(g(n)) &&\textrm{ if } \liminf_{n \rightarrow \infty} f(n)/g(n) >0; \\
f(n) &= \omega(g(n)) &&\textrm{ if } \lim_{n \rightarrow \infty} f(n)/g(n) =\infty; \\
f(n) &= \Theta(g(n)) &&\textrm{ if } f(n) = \O(g(n)) \textrm{ and }  f(n) = \Omega(g(n)).
\end{alignat*}

\noindent
Depending on the number of points compared to the dimension, we derive different estimates.

\begin{theorem} \label{thm1}
For every $p>0$,
\[
M_n^p(S^{d-1})  =   n \mu_{d,p} + o(n d^{-p/2})
\]
as $d, n \rightarrow \infty$ and $n = \omega(d^{1 + p} \log d)$.

Furthermore, for $0<p \leq 2$,
\[
M_n^p(S^{d-1}) = \Theta (n d^{-p/2}),
\]
while for $p>2$,
\[
M_n^p(S^{d-1}) = \Omega(n \, d^{-p/2}) \textrm{ and }M_n^p(S^{d-1}) = \O(n \,d^{-1})
\]
holds, as $d, n \rightarrow \infty$ and $n \geq d$.
\end{theorem}

For special values of $p$ and $d$, stronger results may be proved. In order to discuss the case $p=2$, we introduce the following notion: $\om_n = \{u_1, \ldots, u_n \} \subset S^{d-1}$ is an {\em isotropic set of unit vectors} if
\begin{equation*}
 \sum_{i=1}^n u_i \otimes u_i = \frac n d I_d,
\end{equation*}
where $I_d$ is the identity operator on $\R^d$. Isotropic sets of unit vectors are also called {\em unit norm tight frames}, see e.g. \cite{BF03}.

\begin{theorem} \label{thm2}
For every $d \geq 1$ and $n \geq d$,
\[
M_n^2(S^{d-1}) = \frac n d,
\]
and the extremal $\om_n$ configurations are exactly the isotropic sets of unit vectors.
\end{theorem}

For $p=1$, Theorem~\ref{thm1} provides the exact asymptotics: $M^1_n(S^{d-1}) = \Theta( n d^{- 1/2})$. By the following fact, this also provides an estimate to a quantity involving sign sequences:

\begin{prop} \label{prop3} For any set of unit vectors $\om_n = \{u_1, \ldots, u_n\} \subset S^{d-1}$,
\[
\max_{\varepsilon \in \{-1,1\}^n}\abs{\sum_{i=1}^n\varepsilon_iu_i} = M^1(\om_n).
\]
\end{prop}

\noindent
As a consequence of Proposition~\ref{prop3} and Theorem~\ref{thm1}, we immediately obtain

\begin{theorem}\label{thm4}
For every $d \geq 1$ and $n \geq d$,
\begin{equation}\label{minmaxsign}
\min_{(u_i)_1^n \subset S^{d-1}} \max_{\eps \in \{ \pm 1 \}^n} \abs{\sum_{i=1}^n \eps_i u_i} = \Theta \left( \frac n {\sqrt{d}} \right).
\end{equation}
\end{theorem}

Finally, we discuss the $\ell_p$-polarization constants of the unit circle.

\begin{prop} \label{prop5}
For $d=2$ and $0<p \leq 1$ as well as for $p = 2, 4, \ldots, 2n-2$, $M^p(\om_n)$ is minimized by vector sets which are equally distributed on the half-circle. For $p = 2, 4, \ldots, 2n-2$, the potential function of any extremal configuration is constant on $T$, whereas for $0<p \leq 1$,
\begin{equation*}
M^p_n(S^1) = \sum_{k =1}^n \abs{\cos \left(  \frac {k \pi}  n - \frac \pi {2n }\right)}^p\,
\end{equation*}
for even values of $n$, and
\begin{equation*}
M^p_n(S^1) = \sum_{k =1}^n \abs{\cos \left(  \frac {k \pi}  n \right)}^p\,
\end{equation*}
for odd $n$.
\end{prop}

\section{General asymptotics}

\begin{proof}[Proof of Theorem~\ref{thm1}]
We start with the lower bound, which holds for every $p>0$ and $n \geq d$. Let $\omega_n = \{ u_1, \ldots, u_n \} \subset S^{d-1}$ be fixed.
Note that
\begin{align*}
M^p(\om_n) = \max_{v \in S^{d-1}}\sum_{i=1}^n\abs{\inner{v, u_i}}^p
& \geq \E_v\left[\sum_{i=1}^n\abs{\inner{v, u_i}}^p\right] \\
&= \sum_{i=1}^n\E_v\abs{\inner{v, u_i}}^p \\
&= n \E_v\abs{\inner{v, u_1}}^p,
\end{align*}
where the expectation is taken as $v$ being selected uniformly at random from the sphere. By a standard calculation, we obtain that
\begin{align*}
\E_v\abs{\inner{v, u_1}}^p = \frac{2}{\Beta\left(\frac{1}{2},\frac{d-1}{2}\right)}\int_0^1 t^p \left(1 - t^2\right)^{(d - 3)/2} \dif t = \mu_{d,p} 
\end{align*}
and by the previous arguments,
\begin{equation}\label{lowerbound}
M_n^p(S^{d-1}) \geq n \mu_{d,p}.
\end{equation}

Next, we show that this bound is asymptotically correct when $n$ is large, or when $0<p\leq 2$. First, assume that $n = \Omega(d^{1 + p} \log d)$, and select an independent, random uniform sample
$\om_n = \{u_1, \ldots, u_n\}$  from $S^{d-1}$. We will show that with positive probability, $M^p (\om_n)$ is of order $O(n \, d^{- p/2})$.

For conciseness, let
\begin{equation}\label{def_ffunct}
f(v) = \sum_{i=1}^n \abs{\inner{v,u_i}}^p.
\end{equation}

\noindent
Observe that $f$ is $n p$-Lipschitz, since
\begin{align*}
\abs{f(v) - f(w)} &= \sum_{i=1}^n\left(\,\abs{\inner{v,u_i}}^p - \abs{\inner{w,u_i}}^p\right)\\
&\leq \sum_{i=1}^n\abs{\,\abs{\inner{v,u_i}}^p - \abs{\inner{w,u_i}}^p}\\
&\leq p \sum_{i=1}^n\abs{\,\abs{\inner{v,u_i}} - \abs{\inner{w,u_i}}} \leq p\sum_{i=1}^n\abs{\inner{v-w,u_i}}\\
&\leq p \sum_{i=1}^n\abs{v-w} = np\abs{v-w}.
\end{align*}

On the other hand, for any fixed $v \in S^{d-1}$,
\begin{equation}\label{fexp}
\E f(v) = n \mu_{d,p},
\end{equation}
where the expectation refers to the choice of the random base system $\omega_n$. Moreover, since
 $ 0 \leq \abs{\inner{v,u_i}}^p \leq 1$, Hoeffding's inequality and \eqref{fexp} yields that for any fixed $v \in S^{d-1}$ and $t>0$,
\begin{equation}\label{hoeffding}
\PP\left(\,\abs{f(v) - n \mu_{d,p}} > t\right) < 2e^{ - \frac{2t^2}{ n}}.
\end{equation}
We are going to bound the maximum of $f(v)$ on $S^{d-1}$ by pinning it down at the points of a $\delta$-net and then exploiting the Lipschitz property. It is well known (see e.g. \cite{B97}) that there exists a $\delta$-net (with respect to the Euclidean metric) in $S^{d-1}$ with at most $(4 / \delta)^d$ points. Let $D$ be such a $\delta$-net. Choose $v^* \in S^{d-1}$ be such that
\[
 f(v^*) = M^p(\om_n) = \max_{v \in S^{d-1}} f(v).
\]
Since $v^*$ must be within $\delta$ of some $w \in D$ and $f$ is $np$-Lipschitz, we have that $ | f(w) -  M^p(\om_n) | \leq  \delta n p$. Then, the union bound and \eqref{hoeffding} gives that for every $\lambda >0$,
\begin{align*}
    \PP\left(\,\abs{ M^p(\om_n) - n \mu_{p,d}} > \lambda\right) &\leq \sum_{w \in D}\PP\left(\,{\abs{f(w) - n \mu_{p,d}} > \lambda - \delta n  p}\right)\\
    &\leq 2 \left(\frac{4}{\delta}\right)^{d}e^{- \frac{ 2 (\lambda - \delta n  p)^2}{n}}.
\end{align*}
Setting $\delta = \frac{\lambda}{2 n p}$, the bound simplifies to
\[
\PP\left(\,\abs{ M^p(\om_n) - n \mu_{p,d}} > \lambda\right) \leq 2 \left( \frac{8 n p} {\lambda} \right)^d e^{- \frac{\lambda^2}{2n}}
\]
Take  $\lambda = c n d^{-p/2}$ with some constant $c>0$. Then
\begin{align*}
\PP\left(\,\abs{ M^p(\om_n) - n \mu_{p,d}} > c n d^{- p /2}\right) \leq  {c'}^{\,d} d^{\frac{dp}{2}}e^{- \frac {c^2}{2} n d^{-p}}
\end{align*}
with $c' = 10/c$. Taking logarithm shows that the above probability is guaranteed to be less than one if
\[
n > \frac {2 \log c'}{c^2} \, d^{1 + p} + \frac {p} {c^2} \, d^{1+p} \log d.
\]
Therefore, when $n = \om(d^{1+p} \log d)$, we obtain that
\[M^p_n (S^{d-1})= n \mu_{d,p} + o(nd^{-p/2})= \Theta(nd^{-p/2}).\]

Let us turn to the estimates valid for smaller values of $n$. The lower bound~\eqref{lowerbound} still holds, so we only have to prove the upper estimates.
Without changing the asymptotic bounds, we may assume that $n = k d$. Let $\omega_n$ consist of $k$ copies of an orthonormal basis of $\R^d$. Then,
\[
\max_{v \in S^{d-1}} U^p (\om_n , v ) = k \, \max_{|v|=1} \sum_{i=1}^d |v_i|^p =
\begin{cases}
 k d^{1 - p/2}  & \text{for } 0< p \leq 2 \\
k & \text{for } 2 <p,
\end{cases}
\]
which implies the upper bounds for arbitrary $n \geq d$.
\end{proof}

\medskip

\noindent
{\bf Remark.} The following construction gives a slightly stronger estimate for a small number of points, when $0<p\leq 2$.
Let $c_{n,d,p}$ be the infimum of all constants $c \in \R$ satisfying
\[
M^p_n(S^{d-1}) \leq c\, nd^{-p/2}.
\]
Let $H$ and $H^\perp$ be two orthogonal, $d$-dimensional linear subspaces in $\R^{2d}$. Take $\om_n$ and $\om_n^\perp$ to be $n$-element vector sets in $H$ and $H^\perp$, respectively, with $M^p_n(S^{d-1}) =M^p(\om_n) =M^p(\om_n^\perp)$. Let $\om_{2n} = \om_n \cup \om_n^\perp \subset \R^{2d}$. Then

\begin{align*}
    M^p(\om_{2n}) &= \max_{\substack{v \in H,\, v^\perp \in H^\perp\\\abs{v}^2 + |{v^\perp}|^2 = 1}} \left( U^p(\om_n, v)  +  U^p(\om_n^\perp, v^\perp)\right)\\
    &\leq \max_{\abs{v}^2 + |{v^\perp}|^2 = 1} \left(\,\abs{v}^p + |v^\perp|^p\,\right) c_{n,d,p} \, nd^{-p/2} \\
    &= \frac{2}{2^{p/2}} \, c_{n,d,p} \, nd^{-p/2} = c_{n,d,p} \, (2n)(2d)^{-p/2}.
\end{align*}
Thus, $c_{2n,2d,p} \leq c_{n,d,p}$. Using the fact that for $d=0$, $M^p_n(S^0) = n = nd^{-p/2}$, it follows that for $a,b \in \mathbb{N}, a \geq b$, we have $c_{2^a,2^b,p} \leq 1$. Moreover, for $2^a < n < 2^{a+1}$, it is easy to see that $M^p_n(S^{d-1}) \leq 2M^p_{2^a}(S^{d-1})$ (by taking the vectors of $\om_n$ once or twice).  Likewise, for $2^b < d < 2^{b+1}$, we know that $M^p_n(S^{d-1}) \leq 2^{p/2}M^p_{n}(S^{2^b-1})$ by keeping the optimal vectors from the $d=2^b$ case. Therefore, $c_{n,d,p} \leq 2^{p/2}$ for all $n \geq d$.

\section{Isotropic vector sets: $p=2$}

\begin{proof}[Proof of Theorem~\ref{thm2}.] Let $\om_n = \{ u_1, \ldots, u_n \} \subset S^{d-1}$. Introduce the {\em frame operator}
\[
A= \sum_{i=1}^n u_i \otimes u_i,
\]
where $u \otimes v = u v^\top$ denotes the tensor product of the two vectors. Then for any vector $v \in S^{d-1}$,
\[
v^\top A v =  \sum_{i=1}^n \la v, u_i \ra^2.
\]
Therefore, $\max_{v \in S^{d-1}} \sum \la v, u_i \ra^2$ is attained at the eigenvector of norm 1 of $A$ belonging to the maximal eigenvalue. Since $\tr A = n$, we obtain that
\[
M^2(\om_n) \geq \frac n d,
\]
and equality holds if and only if $\sum_{i=1}^n u_i \otimes u_i = \frac n d I_d$, that is, if $\om_n$ is an isotropic vector system.
\end{proof}

Isotropic vector sets also arise in different contexts: in frame theory, they are called {\em unit norm tight frames} or UNTF's, while in the context of John's theorem, their rescaled copies provide a {\em decomposition of the identity}. A characterization of them was first given by Benedetto and Fickus \cite{BF03} (for a simplified proof, see \cite{A14}): they showed that a set of $n$ unit vectors form an isotropic set if and only if they are the minimizer of the {\em frame potential} among $n$-element vector sets in $S^{d-1}$, defined by
\[
\textrm{FP}(\om_n) = \sum_{i,j} \abs{\inner{u_i,u_j}}^2.
\]
In particular, it follows that $n$-element isotropic sets of $d$-dimensional unit vectors exist for every $n \geq d$.

For $d=2$ and $d=3$, the characterization may be simplified by utilizing the connection with complex numbers. Goyal et al. \cite{GKK01} showed that in $\R^2$, isotropic sets of unit vectors correspond to sequences $\{z_i\}_{i=1}^n \subset \C$ satisfying $\abs{z_i} = 1$ and
\[
\sum_{i=1}^n z_i^2 = 0,
\]
where the unit circle $S^1$ of $\R^2$ is identified with the complex unit circle $T$. For $d=3$, Benedetto and Fickus [2] provide a correspondence between isotropic vector sets and
sequences $\{z_i\}_{i=1}^n \subset \C$ satisfying $\abs{z_i} \leq 1$ and
\[
\sum_{i=1}^n \abs{z_i}^2= \frac{2}{3}n, \qquad \sum_{i=1}^n z_i^2 = 0, \qquad \sum_{i=1}^n z_i\sqrt{1 - \abs{z_i}^2} = 0.
\]
Here, each point in $S^2$ is identified with its projection onto the unit disc of the complex plane.

\section{Sign sequences: $p=1$}
\begin{proof}[Proof of Proposition~\ref{prop3}]

First, we show that $\max_{\varepsilon \in \{-1,1\}^n}\abs{\sum_{i=1}^n\varepsilon_iu_i} \leq M^1(\om_n).$ Indeed, let $\eps$ be an arbitrary sign sequence, and define
\[
z = \sum_{i=1}^n \eps_i u_i.
\]
Then
\begin{align*}
|z|^2 = \left| \sum \eps_i u_i  \right|^2 & = \sum_{i, j} \eps_i \eps_j \la u_i, u_j \ra \\
&= \sum\eps_i \la z, u_i \ra \leq \sum ^ n|\la z, u_i \ra|
\end{align*}
which shows that
\[
\left|\sum \eps_i u_i\right| \leq U^1 \left(\om_n, \frac {z}{|z|} \right) \leq M^1(\om_n).
\]

For the reverse direction, introduce the function $f(v) = \sum_{i=1}^n \abs{\inner{v,u_i}}$ defined on $S^{d-1}$ as in \eqref{def_ffunct}. Applying Lagrange multipliers implies that those critical points of $f$ on $S^{d-1}$ where $f$ is differentiable satisfy
\begin{equation*}
v =  \frac {\sum  \eps_i u_i}{|\sum \eps_i u_i|}
\end{equation*}
with $\eps_i = \sgn{\inner{v, u_i}}$. By taking inner products of both sides with $v$ we obtain that $|\sum \eps_i u_i| = U^1(\om_n, v)$.

Therefore, we only have to rule out the existence of maximizers of $f$ at non-differentiable points.
Assume on the contrary that $v \in S^{d-1}$ is a maximizer with $\la v, u_j \ra = 0$, where $1 \leq j \leq k$, and $\abs{\inner{v, u_j}}>0$ for $k < j \leq n$. Then for $\delta \in \R$ with sufficiently small absolute value,
 \begin{align*}
f(v + \delta u_1) &= \sum_{i=1}^n\abs{\inner{v + \delta u_1,u_i}} \\
&= |\delta| \sum_{i=1}^k\abs{\inner{u_1,u_k}} + \sum_{i=k+1}^n \abs{ \la v, u_i \ra + \delta \la u_1, u_i \ra}\\
&= f(v) + \delta   \sum_{i=k+1}^n \sgn{\inner{v,u_i}} \cdot \inner{u_1,u_i}  + |\delta| \sum_{i=1}^k\abs{\inner{u_1,u_k}}
.
\end{align*}
Here, $\sum_{i=1}^k\abs{\inner{u_1,u_k}} \geq 1$, and thus, for sufficiently small but non-zero $\delta$ whose sign agrees with that of $\sum_{i=k+1}^n \sgn{\inner{v,u_i}} \cdot \inner{u_1,u_i} $, we would obtain that
\[
f\left(\frac{v + \delta u_j}{|v + \delta u_j|}\right) \geq \frac{f(v) + |\delta|}{\sqrt{1 + \delta^2}}>f(v),
\]
which contradicts the maximality of $v$.
\end{proof}

Sign sequences arise in several topics, most prominently in the context of discrepancy theory (see, for example, the famous conjecture of Koml\'os~\cite{S87}). Note, however, a fundamental difference: in that setting, one would like to {\em minimize} the norm of $\sum \eps_i u_i$, whereas here, the goal is to find the {\em maximizers}. The ``dual'' question of Theorem~\ref{thm4} was asked by Dvoretzky~\cite{D63} in 1963: Determine
\[
\max_{(u_i)_1^n \in S^{d-1}} \min_{\eps \in \{ \pm 1 \}^n} \abs{\sum_{i=1}^n \eps_i u_i} .
\]
Various related games had been studied by Spencer~\cite{S77}. B\'ar\'any and Grinberg~\cite{BG81} proved a stronger result which implies an $O(d)$ upper bound on the above quantity.

More related to the present question is Bang's lemma~\cite{B51}, which arose in the context of the well-known plank problem. Its simplest form \cite{B01} states the following: {\em If $u_1, \ldots, u_n$ are unit vectors in $\R^d$, and the signs $\eps_i = \pm 1$ are chosen so as to maximize the norm $| \sum_1^n \eps_i u_i|$, then $|\la u_k, \sum_1^n \eps_i u_i \ra| \geq 1$ holds for every $k$.} Note, however, that this only implies
\[
\min \max \left|\sum_1^n \eps_i u_i\right| \geq \sqrt{n}.
\]
 The same estimate follows by taking the average of $|\sum \eps_i u_i|^2$ over all possible sign sequences.

It remains an open question to determine the extremal point configurations of \eqref{minmaxsign}. In general, we have very little information about the extremizers, and a complete description of them can only be hoped for in a few special cases. For $n=d$, the above averaging argument yields that the extremum is uniquely achieved by the vectors $e_i$ of an orthonormal basis, which satisfy $\min \max |\sum_1^n \eps_i e_i|  = \sqrt{d}$.  For $n = d+1$, natural intuition and numerical experiments suggest that each extremal configuration is, up to sign changes, the union of the vertex set of an even dimensional regular simplex and an orthonormal basis of the orthogonal complement of its subspace. The following conjecture was stated in a slightly incorrect form in \cite{BFMK18} and has been corrected by \cite{P19}.

\begin{conj}\label{simplexconj}
For any $d \geq 1$, and for any configuration of $d+1$ unit vectors $u_i, \ldots, u_{d+1} \in S^{d-1}$, there exists a sequence of signs $\eps \in \{\pm 1\}^{d+1}$ so that
\[
\left| \sum_{i = 1} ^ {d+1} \eps_i u_i \right| \geq \sqrt{d+2}.
\]
Moreover, the above estimate is sharp if and only if, up to sign changes, $(u_i)_1^{d+1}$ is the union of the vertex set of a regular simplex centered at the origin in a subspace $H$, and an orthonormal basis of $H^\perp$, where $H$ is an even dimensional linear subspace of $\R^d$.
\end{conj}

\section{Planar case: equidistributed sets} \label{sec_planar}

In the plane, finding $M^p(\om_n)$ is equivalent to  maximizing the sum of the $p$th powers of the Euclidean distances from a variable unit vector to $n$ fixed unit vectors via the following transformation.
Identify $S^1$ with the complex unit circle $T$, and let $u_i = e^{i \alpha_i}$, $v = e^{i \phi}$.
Introduce $\u_i = e^{i 2 \alpha_i} = u_i^2$ and $\vv = e^{i (2 \phi + \pi)} = - e^{i 2 \phi}$. Then
\begin{equation}\label{distcomplex}
|\vv - \u_i | = 2 \abs{\sin{ \frac {2 ( \phi + \pi/2) - 2\alpha_i }2 }  } = 2 | \cos( \alpha_i - \phi)| = 2 \abs{ \la v, u_i\ra}.
\end{equation}
Therefore, $M^p(\om_n)$ may be obtained by finding the point $\vv \in S^1$ for which $\sum \abs{\u_i - \vv}^p$ is maximal.

Accordingly, we introduce the following quantities for an $n$-point configuration $\omega_n = \{ z_1, \dots, z_n \} \subset T$:
\begin{align*}
    \widetilde{U}^p(\omega_n, z) &= \sum_{i=1}^n\abs{z - z_i}^p\\
    \M^p(\omega_n) &= \max_{z \in T} \widetilde{U}^p(\omega_n,z)\\
    \M^p_n &= \min_{\omega_n \in T^n} \M^p(\omega_n).
\end{align*}
Analogues of the above notions with negative $p$ are called the Riesz potential and polarization quantities, and have been extensively studied before, see e.g. \cite{ES13} for general results in that direction.

By \eqref{distcomplex}, $\M^p_n = 2^p M^p_n(S^1)$, and thus Theorem~\ref{thm1} implies the lower bound
\begin{equation}\label{MMtilde}
\M^p_n \geq 2^p \cdot n \mu_{2,p} = n \cdot \widetilde{\mu}_p,
\end{equation}
where
\[
\widetilde{\mu}_p = 2^p \mu_{2,p} =  \frac{2^p\Gamma\left(\frac{p+1}{2}\right)}{\sqrt{\pi}\,\Gamma\left(\frac{p}{2} + 1\right)}
= \frac{\Gamma\left(p + 1\right)}{\Gamma\left(\frac{p}{2} + 1\right)^2} = \binom{p}{p/2},
\]
using the Legendre duplication formula and the natural extension of the binomial coefficient to non-integers.

The above notions have been studied by Stolarsky \cite{S75a, S75b}, who determined $\M^p(\omega_n^*)$ for $0<p<2n$,  where $\om_n^*$ is an {\em equidistributed set} on $T$:
\begin{equation*}
\omega_n^* = \{ 1, \xi, \xi^2, \dots, \xi^{n-1} \},
\end{equation*}
where $\xi = e^{i2\pi/n}$. He also determined $\M^p_n$ for $n=3$ and $0 <p \leq 2$. Nikolov and Rafailov \cite{NR12} determined the value $\M^p_n$ for $n=3$ and arbitrary $p>0$ and also discussed the critical points of  $\widetilde{U}^p(\omega_n^*, z)$ on $T$. They showed that if $p$ is an even integer with $ 2 \leq p \leq 2n - 2$, then $\widetilde{U}^p(\omega_n^*, z)$ is constant on $T$. Moreover, they proved \cite{NR13} that this property (holding for all even integer exponents between $2$ and $2n$) characterizes equidistributed sets. They conjectured that the condition holding solely for $p = 2n - 2$ is already sufficient for characterization. This was verified by Bosuwan and Ruengrot \cite{BR17} (for the case $\om_n \subset T$, which we assumed anyway). The authors also proved that for $ p = 2, 4, \ldots, 2n-2$, $\M^p_n$ is attained at the configurations $\om_n$ which satisfy
\[
\sum_{z \in \om_n} z^j = 0
\]
for every $j = 1, 2, \ldots, n-1$.

On the other hand, Hardin, Kendall and Saff \cite{HKS13}, proving a conjecture of the first named author formulated in \cite{ABE12}, proved the polarization optimality of equidistributed sets on the unit circle for convex potentials. Recently, their result has been extended to more general settings~\cite{FNR18}.

\begin{proof}[Proof of Proposition~\ref{prop5}]
By \eqref{distcomplex}, finding the polarization constants is equivalent to maximizing the quantity $\sum \abs{\u_i - \vv}^p$.

First, we assume $0 < p \leq 1$. Let $g(t) = - \abs{\sin(t/2)}^p + 1.$ Then $g$ is non-negative, non-increasing and strictly convex on $[0, 2\pi]$. Moreover,
\[
\frac 1 2 \sum \abs{\vv - \u_i}^p = - \sum g(\psi - \beta_i) + n,
\]
where $\vv = e^{i \psi }$ and $\u_i =  e^{i \beta_i }$. Therefore, $\M^p(\om_n)$ is attained when $\sum g(\psi - \beta_i)$ is minimized. Theorem~1 of \cite{HKS13} implies that $\M^p_n$ is achieved by equidistributed points sets; moreover, these are the only optimizers. Accordingly, the lines spanned by an optimal configuration for $M^p_n(S^1)$ are evenly spaced. It is easy to check \cite{S75b} that for such a configuration, the maximum of the potential function $U^p(\om_n, \,\cdot \,)$ is attained at one of the base points for odd $n$, and at the midpoint between two consecutive base points for even $n$.

The case $p = 2, 4, \ldots, 2n-2$ is discussed in \cite{BR17}, Theorem 2. We also give a short proof here. It was shown in \cite{NR12} that for these values of $p$, $\U^p(\omega_n^*, z)$ is constant on $T$. Therefore, for any $n$-point configuration $\om_n$ on~$T$,
\begin{align*}
\M^p(\om_n) &\geq \frac 1 n \sum_{z \in \omega_n^*} \U^p(\omega_n, z)\\
&= \frac 1 n \sum_{v \in \omega_n} \U^p(\omega_n^*, v)\\
&= \M^p(\om_n^*)\, . \qedhere
\end{align*}
\end{proof}

\noindent

For equally distributed point sets, it was proven by Stolarsky \cite{S75b} and by Nikolov and Rafailov \cite{NR12} that $\M^p(\omega_n^*) = \max_{z \in T} \sum_{k= 0}^{n-1} |z - \xi^k|^p$ is (not necessarily uniquely) attained at $z$ which is, depending on $p$, either one of the  base points $\xi^k$ or is the midpoint between two consecutive base points. More precisely, introduce the {\em positive-exponent Riesz energy} of $\om_n \subset T$ defined by
\[
E^p(\omega_n) = \sum_{j, k=1}^n \abs{z_j - z_k}^p
\]
(note that in the previous articles related to Riesz energies, the exponent is taken to be $-p$, therefore the above quantity becomes the negative exponent Riesz energy). For brevity, let $E^p_n = E^p(\om_n^*)$. Theorem 1.2. of \cite{S75b} states that for $0 < p < 2n$, taking $m = \lfloor p/2 \rfloor$,
\begin{equation} \label{Mp1}
\M^p(\omega_n^*) = \begin{cases}
\frac{E_n^p}{n}, & m \text{ odd}\\
\frac{E_{2n}^p}{2n} - \frac{E_n^p}{n}, & m \text{ even}.
\end{cases}
\end{equation}
Furthermore, for $p \geq 2n$, Theorem 2 of \cite{NR12} implies that
\begin{equation} \label{Mp2}
\M^p(\omega_n^*) = \begin{cases}
\frac{E_n^p}{n}, & n \text{ even}\\
\frac{E_{2n}^p}{2n} - \frac{E_n^p}{n}, & n \text{ odd}.
\end{cases}
\end{equation}
The asymptotic expansion of $E_n^p$ was given by Brauchart, Hardin and Saff~\cite{BHS09}:
\[
E_n^p = n^2\widetilde{\mu}_p + \mathcal{O}(n^{1-p}), \quad n \to \infty.
\]
This, along with \eqref{MMtilde}, \eqref{Mp1}, \eqref{Mp2}, and the fact $ \M_n^p  \leq \M^p(\omega_n^*)$, implies that $
\M_n^p \sim n\widetilde{\mu}_{p} = n \binom{p}{p/2}
$ as $n \to \infty$.

Proposition~\ref{prop5} shows that for integer exponents $p$ with $0 < p < 2n$, $\M_n^p = \M^p(\omega_n^*)$.
For these exponents, we provide the explicit value of $E_n^p$ (and, by \eqref{Mp1} and \eqref{Mp2}, of $\M^p(\omega_n^*)$) by a combinatorial argument. If $p = 2m$ for some integer $m<n$,
\begin{align*}
E_n^p &= \sum_{j = 0}^{n-1}\sum_{k=0}^{n-1} \abs{\xi^j - \xi^k}^{2m} = \sum_{j = 0}^{n-1}\sum_{k=0}^{n-1} \abs{1 - \xi^{k-j}}^p = n \sum_{j=0}^{n-1} \abs{1 - \xi^j}^{2m}\\
&= n \sum_{j=0}^{n-1}(1 - \xi^j)^m(1 - \xi^{-j})^m.
\end{align*}
Using binomial expansion gives
\begin{align*}
E_n^p &= n \sum_{j=0}^{n-1}\sum_{r,s=0}^{m} \binom{m}{r} \binom{m}{s} (-1)^{r+s} (\xi^j)^{r-s}\\
&= n \sum_{r,s=0}^m \binom{m}{r} \binom{m}{s} (-1)^{r+s}\sum_{j=0}^{n-1}(\xi^{r-s})^j\\
&= n^2 \sum_{r}^m \binom{m}{r}^2 = n^2 \binom{2m}{m} = n^2 \widetilde{\mu}_p.
\end{align*}
On the other hand, assume that $p$ is odd with $0 < p < 2n$. Noting that for $t \in [0, 2\pi)$,
\[
\abs{1 - e^{it}} = ie^{-it/2}(1 - e^{it}),
\]
it follows that
\begin{align*}
E^p_n &= n \sum_{j=0}^{n-1} \abs{1 - \xi^j}^p = n \, i^p\sum_{j=0}^{n-1}\xi^{-pj/2}(1-\xi^j)^p\\ &= n \, i^p \sum_{j=0}^{n-1}\sum_{s=0}^{p}\binom{p}{s}(-1)^s \xi^{(s-p/2)j}\\
&= n \, i^p \sum_{s=0}^{p}\binom{p}{s}(-1)^s \sum_{j=0}^{n-1}\xi^{(s-p/2)j} \, .
\end{align*}
Now, using that $p$ is odd, we have
\begin{align*}
E^p_n &= n \, i^p \sum_{s=0}^{p}\binom{p}{s}(-1)^s \frac{\xi^{n(s-p/2)}-1}{\xi^{s-p/2}-1}\\
&= 2n \, i^p \sum_{s=0}^{p}\binom{p}{s} \frac{(-1)^s}{1-\xi^{s-p/2}},
\end{align*}
since $\xi^{ns} = 1$ and $\xi^{-np/2}=-1$. Observing the symmetry of this sum about $p/2$, we can compute
\begin{align*}
E^p_n &= 4n \, i^{p+1} \, \im\left(\sum_{s=0}^{\lfloor p/2 \rfloor}\binom{p}{s} \frac{(-1)^s}{1 -\xi^{s-p/2}}\right)\\
&= n \, (-1)^{\frac{p-1}{2}} \sum_{s=0}^{p}\binom{p}{s} (-1)^s\cot\left(\left(\frac{p}{2} - s\right)\frac{\pi}{n}\right).
\end{align*}
For $p=1$, this gives $E_n^p = 2n\cot(\frac{\pi}{2n}) \sim \frac{4n^2}{\pi} = n^2 \widetilde{\mu}_{1}$. In general, as $n \to \infty$, the Taylor expansion of the cotangent gives
\[
    E_n^p = n\,\abs{\sum_{s=0}^p \binom{p}{s} (-1)^s \frac{n}{\pi(\frac{p}{2}-s)} + \mathcal{O}(1/n)} =n^2 \widetilde{\mu}_{p} + \mathcal{O}(1),
\]
where the second equality follows from a series computation described in Proposition 2.3 of \cite{G07}.

We conclude the paper by restating the following natural conjecture of Bosuwan and Ruengrot \cite{BR17}, which is also supported by our numerical experiments:

\begin{conj}
For any $n \geq 1$, the vector systems achieving $\M^p_n$ are equally distributed on the circle for every $p \in \R^+ \setminus \{ 2, 4, \ldots, 2n - 2 \}$.
\end{conj}

\noindent
{\bf Acknowledgement.} This research was done under the auspices of the Budapest Semesters in Mathematics program. The authors are grateful to I. Bárány for fruitful discussions and to A. Polyanskii for communicating the correct form of Conjecture 1 to us.

\end{document}